\documentclass{article}
\usepackage{amsmath,amsthm,amssymb,amsfonts,latexsym,stmaryrd,mathrsfs}
\usepackage[shortlabels]{enumitem}
\usepackage{cmll,bbm} 
\usepackage[arrow, matrix, tips, curve, graph, rotate]{xy}
\SelectTips{cm}{10}
\newcommand{\cd}[2][]{\vcenter{\hbox{\xymatrix#1{#2}}}}

\newtheorem{thm}{Theorem}[section]
\newtheorem{lem}[thm]{Lemma}

\theoremstyle{definition}
\newtheorem{ex}[thm]{Example}

\newcommand{\pair}[2]{\langle #1 , #2 \rangle}
\newcommand{\arr}{\rightarrow}

\def\lbr{\mathopen{{[\kern-0.14em[}}}   
\def\rbr{\mathclose{{]\kern-0.14em]}}}  

\newcommand{\strarr}{\mbox{$\circ \kern-0.4em \arr$}}
\newcommand{\eps}{\,\varepsilon\,}
\newcommand{\neps}{\; \slash \kern-0.7em \eps \,}

\newcommand{\E}{{\mathcal{E}}}

\newcommand{\Ab}{{\mathbb{A}}}
\newcommand{\Bb}{{\mathbb{B}}}

\newcommand{\forget}[1]{}

\newcommand{\Sm}{{\mathcal{S}}}

\newcommand{\Eff}{\mathcal{E}\kern-0.14em\mathit{ff}}

\newcommand{\schloop}{{< \kern-0.40em  <}}

\begin{document}

\title{\textsc{An Essential Local Geometric Morphism 
               which is not Locally Connected though its Inverse Image Part 
               defines an Exponential Ideal}}
\author{R.~Garner and T.~Streicher}
\date{May 2021}
\maketitle

In \cite{BP} the authors introduced and studied a property of geometric
morphisms between toposes which they called ``molecular'' and nowadays is
usually referred to as ``locally connected''. One of the various
characterizations of this property is that the inverse image part $F$ of the
geometric morphism $F \dashv U \colon \E\to\Sm$ preserves dependent products,
i.e.\ right adjoints to pullback functors. As described below this requirement
is tantamount to $F$ having a fibered or ``indexed'' left adjoint. In
\emph{loc.~cit.}\ the authors also consider the property that $F$ has an
enriched left adjoint, which amounts to $F$ preserving (ordinary) exponentials
and having a left adjoint.

In this note, we exhibit a geometric morphism
$F \dashv U \colon \E\to\Sm$ which satisfies the weaker of these two
sets of conditions, but not the stronger: thus, it is not locally
connected although it is essential (i.e. $F$ has a left adjoint) and
the inverse image $F$ preserves exponentials. In fact, $F$ will be
full and faithful and $U$ will have a right adjoint, i.e.~our
geometric morphism is \emph{local}; moreover, the left adjoint of $F$
will preserve finite products, so that $F$ even exhibits $\Sm$ as an
\emph{exponential ideal} in $\E$.

\section{Preliminaries}

A geometric morphism $F \dashv U \colon \E\to\Sm$ is called \emph{locally
  connected} just when $F$ has a left adjoint $L$ which is fibered or indexed
over $\Sm$. A succinct way of expressing this is as the requirement that
\begin{equation}\label{eq:1}
  \cd{
    {B} \ar[r]^-{f} \ar[d]_{b} &
    {A} \ar[d]^{a} \\
    {FJ} \ar[r]_-{Fu} &
    {FI}
  } \ \ \text{ a pullback} \quad \text{implies} \quad
  \cd{
    {LB} \ar[r]^-{Lf} \ar[d]_{\widehat b} &
    {LA} \ar[d]^{\widehat a} \\
    {J} \ar[r]_-{u} &
    {I}
  } \ \  \text{ a pullback}
\end{equation}
as discussed, for example, in \cite{BP,Ele,Str}. 

When $I$ (and thus also $FI$) is a terminal object, this condition boils down to
the requirement that $L \dashv F$ validates \emph{Frobenius reciprocity}, i.e.\
that for all $I \in \Sm$ and $A \in \E$ the canonical map
$\pair{\widehat{\pi_1}}{L\pi_2} \colon L(FI \times A) \to I \times LA$ is an
isomorphism. One easily checks that local connectedness is equivalent to
Frobenius reciprocity holding not only for the adjunction $L \dashv F$ itself,
but also for each adjunction on slices
$L_I = \Sigma_{\varepsilon_I} \circ L_{/FI} \dashv F_{/I} = F_I$ (where 
$\varepsilon_I$ is the counit of $L \dashv F$ at $I$). Furthermore, by
\cite[Lemma~A.1.5.8]{Ele} an adjunction between cartesian closed categories
validates Frobenius reciprocity if and only if its right adjoint preserves
exponentials.

In fact, for a geometric morphism $F \dashv U \colon \E\to\Sm$ the
following conditions are equivalent:
\begin{enumerate}[(i),itemsep=0em]
\item $F \dashv U$ is locally connected;
\item $F$ preserves dependent products (i.e.\ right adjoints to pullback
  functors);
\item $F_{/I}$ preserves exponentials for all $I\in\Sm$,
\end{enumerate}
as formulated in \cite[Proposition~C3.3.1]{Ele}.
\forget{\footnote{Note that this result
is stronger than what would follow from the discussion above, since
for the implications (ii) $\Rightarrow$ (i) and (iii) $\Rightarrow$
(i) it is not necessary to assume the existence of the left adjoint
$L$; rather, this can be deduced from the fibred adjoint functor
theorem.}}

Some of the notions and remarks above extend to finitely complete
categories. We call an adjunction $L \dashv F \colon \Ab\to\Bb$
between finitely complete categories \emph{stably Frobenius} when it
validates the condition in diagram (\ref{eq:1}). Stably Frobenius
adjunctions with full and faithful right adjoint $F$ are sometimes
called \emph{semi-left-exact reflections}~\cite{CHK}; they are
reflections whose left adjoint preserves pullbacks along morphisms in
the image of $F$.

Further, if $\Ab$ and $\Bb$ are locally cartesian closed then by 
\cite[Lemma~A1.5.8]{Ele} an adjunction $L \dashv F \colon \Ab\to\Bb$ is 
stably Frobenius if and only if condition (iii) above holds, i.e.\ all slices 
$F_{/I}$ preserve exponentials. In fact, in this case we can say slightly more; 
$F$ not only preserves exponentials but \emph{creates} them:
\begin{lem}
  \label{lem:1}
  If $L \dashv F \colon \Ab \rightarrow \Bb$ is a semi-left-exact reflection
  between finitely complete categories, and $\Bb$ is locally cartesian closed, 
  then so too is $\Ab$.
\end{lem}
\begin{proof}
  This is~\cite[Lemma~4.3]{GL}.
\end{proof}

For a reflection $L \dashv F \colon \Ab \rightarrow \Bb$, a stronger condition 
than Frobenius reciprocity is the requirement that $L$ preserve all finite 
products. If $\Bb$ is cartesian closed then, by \cite[Proposition~A4.3.1]{Ele}, 
$L$ preserve all finite products just when the subcategory determined by $F$ is (not 
only exponential-closed but) an exponential \emph{ideal}. The ``stable'' 
version of this condition is that $L$ preserves all pullbacks over objects in 
the image of $F$; in~\cite{CHK} this condition was called 
\emph{having stable units}.

It has been shown in Prop.~10.3 of \cite{LM15} that for essential connected geometric 
morphisms $F \dashv U \colon \E\to\Sm$ the left adjoint $L$ of $F$ has stable units 
if and only if $F \dashv U$ is locally connected and $L$ preserves binary products.

As shown 
in \cite[Proposition~2.7]{JohPLC} if $L \dashv F \dashv U \dashv R \colon \Sm\to\E$ 
with $F$ (and thus also $R$) full and faithful then $\Sm$ is an exponential 
ideal in $\E$ (via $F$) whenever $F$ preserves exponentials and all components 
of the canonical transformation $\theta \colon U \to L$ are epimorphic.
Here, as in \cite{JohPLC,LM15} $\theta_A \colon UA \to LA$ is the unique map 
whose image under $F$ is 
$\eta_A \circ \varepsilon_A \colon FUA \rightarrow A \rightarrow FLA$.
As also shown in \cite{JohPLC}, for a locally connected, hyperconnected and local
geometric morphism $F \dashv U \colon \E\to\Sm$, the left adjoint $L$ of $F$ 
necessarily preserves finite products.

\section{The counterexample}

Any functor $L \colon \Bb \rightarrow \Ab$ between small categories
induces a geometric morphism
$L^\ast \dashv L_\ast \colon \widehat{\Bb} \to \widehat{\Ab}$ between
presheaf categories, where $L^\ast$ and $L_\ast$ are restriction and
right Kan extension along $L$. This geometric morphism is always
essential, since we have $L_! \dashv L^\ast$, where $L_!$ is left Kan
extension along $L$. If $L$ has a fully faithful right adjoint
$L \dashv F \colon \Ab \rightarrow \Bb$, then $L_\ast$ has the fully
faithful right adjoint $F_\ast$, so our geometric morphism is in fact
\emph{local}.

\begin{lem}\label{lem:2}
For reflections $L \dashv F \colon \Ab \rightarrow \Bb$ 
between small finitely complete categories the following assertions hold.
  \begin{enumerate}[(i),itemsep=0em]
  \item[\emph{(i)}] $L^\ast$ is fully faithful;
  \item[\emph{(ii)}] If $L$ preserves finite products, then so does $L_!$, so that $L^\ast$
    exhibits $\widehat{\Ab}$ as an exponential ideal in $\widehat{\Bb}$;
  \item[\emph{(iii)}] If $L^\ast \dashv L_\ast$ is locally connected then $L \dashv F$
    is semi-left-exact.
  \end{enumerate}
\end{lem}
\begin{proof}
  For (i) note that $L^*$ is naturally
  isomorphic to $F_! \colon \widehat{\Ab} \rightarrow \widehat{\Bb}$, and left
  Kan extension along a fully faithful functor is always fully faithful.

  For (ii), since $L$ preserves finite products, $L_!$ preserves finite products
  of representables. Since every presheaf is a colimit of representables, and
  $\times$ preserves colimits in each variable, it follows that $L_!$ preserves
  finite products.

  Finally, for (iii), to say that $L^* \dashv L_*$ is locally connected is to say
  that $L_! \dashv L^*$ is stably Frobenius. Since $L_!$ and $L^* \cong F_!$
  preserve representable objects, the adjunction $L \dashv F$ is also stably
  Frobenius; and since $F$ by assumption is full and faithful, the adjunction
  $L \dashv F$ is in fact semi-left-exact.
\end{proof}

From Lemmas~\ref{lem:1} and~\ref{lem:2} we immediately obtain:
\begin{thm}
  Consider a reflection between small finitely complete categories
  $L \dashv F \colon \Ab \rightarrow \Bb$ for which:
  \begin{itemize}[itemsep=0em]
  \item  $L$ preserves finite products; 
  \item $\Bb$ is locally cartesian
    closed;
  \item  $\Ab$ is \emph{not} locally
    cartesian closed.
  \end{itemize}
  The geometric morphism
  $L^\ast \dashv L_\ast \colon \widehat \Bb \rightarrow \widehat \Ab$ is
  essential and local but not locally connected, i.e.\ $L^*$ does not preserve
  dependent products, although the left adjoint $L_!$ of $L^*$ preserves finite
  products, i.e.\ the full subcategory of $\widehat{\Bb}$ given by the
  image of $\widehat{\Ab}$ under $L^*$ is
  an exponential ideal.
\end{thm}

To complete the construction of our counterexample, it thus suffices to find a
reflection $L \dashv F \colon \Ab \rightarrow \Bb$ with the properties listed
above. In fact, we give two examples of quite different flavour.

\begin{ex}
   In \cite{GL} the authors consider the following situation. Let $\Bb$ be a
  small category equivalent to the category of finite reflexive graphs and
  morphisms between them, and let $F \colon \Ab \to \Bb$ be the inclusion of the
  full subcategory of finite preorders. Example~3.9 of \emph{loc.~cit.}~shows
  that $\Ab$ is an exponential ideal in $\Bb$, so that $L$ preserves finite
  products. Of course, as a category of finite presheaves, the category $\Bb$ is
  locally cartesian closed; however, as also noted in \emph{loc.~cit.}, the
  category $\Ab$ is not so. So by the corollary above, the geometric morphism
  $L^* \dashv L_*$ is essential and local, with inverse image giving rise to
  an exponential ideal, but is not locally connected. 
\end{ex}

\begin{ex}
  \label{ex:2}
   A further counterexample arises from realizability models of dependent
  type theories as studied in \cite{SemTT}. For our purposes we may
  restrict to the case where the underlying partial combinatory algebra
  is the \emph{First Kleene Algebra} corresponding to the standard notion
  of computation on natural numbers. 
 
  For $\Bb$ one takes the category of modest sets and for $\Ab$ the
  full reflective subcategory on $\neg\neg$-closed subobjects of
  powers of $N$, the natural numbers object of $\Bb$. The category $\Bb$ 
  is locally cartesian closed and $\Ab$ forms a full reflective subcategory 
  of $\Bb$ which is an exponential ideal since it is a model of the Calculus 
  of Constructions as verified in Chapter 2 of \cite{SemTT}. Moreover, as 
  follows from (the proof of) Theorem 2 in the Appendix of \cite{SemTT}, the 
  category $\Ab$ is not a sub-locally-cartesian-closed-category of $\Bb$.

  For sake of completeness we describe the counterexample explicitly
  but refer the reader for its verification to \emph{loc.cit}. Let
  $I = N^N$ and $A$ be the $\neg\neg$-closed subobject of $N^N$ consisting 
  of those $f$ with $f(0) \leq 1$ and $f(0)=0$ if and only if $f(n+1) = 0$ 
  for all $n$. Let $g \colon A \to I$ send $f$ to $\lambda n. f(n+1)$ and 
  $h : B \to A$ be the projection $\pi \colon A \times N \to A$. Obviously, 
  both $g$ and $h$ are maps in $\Ab$ but the domain of the dependent
  products $\Pi_g h$ taken in $\Bb$ is not isomorphic to an object in $\Ab$ 
  as can be shown using the Kreisel-Lacombe-Shoenfield Theorem well known 
  from computability theory.
\end{ex}

Note that in both examples the geometric morphism $L^* \dashv L_*$ is not hyperconnected
since $L^*$ is not full on subobjects. The reason is that there are sieves $S$ in $\Bb$ on
some $FA$ such that some $u : B \to FA$ is in $S$ but its reflection $FLu$ is not in $S$.

In \cite{HR20} one finds an example of an essential hyperconnected and local geometric 
morphism which, however, fails to be locally connected. Though their counterexample and 
ours are quite different in nature, together they seem to suggest that an essential 
hyperconnected local geometric morphism need not be locally connected even if its inverse 
image part is an exponential ideal. This would provide a (negative) answer to the question
raised immediately after \cite[Corollary~10.4]{LM15}. 

A positive answer to this question would mean that the inverse image part of a 
hyperconnected and local geometric morphism necessarily preserves dependent products
whenever it preserves exponentials. This, however, is most unlikely since it would mean
that dependent functions spaces are definable in first order intuitionistic logic 
from ordinary function spaces.

\end{document}